\documentclass{amsart}
\usepackage{amssymb}
\usepackage{mathrsfs}
\usepackage{amsmath}
\usepackage{amsfonts}
\usepackage{pifont}
\usepackage{graphicx,amssymb,mathrsfs,amsmath}
\usepackage{tocvsec2}
\usepackage{color}

\usepackage{tikz}
\newtheorem{thm}{Theorem}[section]
\newtheorem{cor}[thm]{Corollary}
\newtheorem{lem}[thm]{Lemma}
\newtheorem{prop}[thm]{Proposition}
\theoremstyle{definition}
\newtheorem{defn}[thm]{Definition}
\newtheorem{example}[thm]{Example}
\newtheorem{rem}[thm]{Remark}
\newtheorem{ques}[thm]{Question}

\numberwithin{equation}{section}
\newcommand{\norm}[1]{\left\Vert#1\right\Vert}

\newcommand{\eps}{\varepsilon}

\newcommand{\card}{\operatorname{card}}
\newcommand{\ldens}{\operatorname{\underline{dens}}}
\newcommand{\udens}{\operatorname{\overline{dens}}}

\newcommand{\budens}{\operatorname{\overline{Bd}}}
\newcommand{\bldens}{\operatorname{\underline{Bd}}}

\newcommand{\spa}{\operatorname{span}}      %% Func. env. lineal: \spa

\begin{document}
\title[Reiterative distributional chaos on Banach spaces]{Reiterative distributional chaos on Banach spaces}

\author{Antonio Bonilla}
\address{Departamento de An\' alisis Matem\' atico,
Universidad de la Laguna, 38271, La Laguna (Tenerife), Spain}
\email{abonilla@ull.es}

\author{Marko Kosti\' c}
\address{Faculty of Technical Sciences,
University of Novi Sad,
Trg D. Obradovi\' ca 6, 21125 Novi Sad, Serbia}
\email{marco.s@verat.net}

{\renewcommand{\thefootnote}{} \footnote{
\\ \text{  }  \ \    {\it Key words and phrases.} Reiterative distributional chaos; distributional chaos; hypercyclicity; Li-Yorke chaos; irregular vectors.
\\  \text{  }  \ \ The first author was supported by MINECO and FEDER,
Project MTM2016-75963-P.The  second author was partially supported by grant 174024 of Ministry of Science and Technological Development, Republic of Serbia.}}

\begin{abstract}
If we change the upper and lower density in the definition of distributional chaos of a continuous linear operator on Banach space by the Banach upper and Banach lower density, respectively, we  obtain Li-Yorke chaos. Motivated by this fact, we introduce the notions of reiterative distributional chaos of types $1$, $1^+$ and  $2$ for continuous linear operators on Banach spaces, which are  characterized in terms of the existence of an irregular vector  with additional properties. Moreover, we study its relations with other dynamical properties and give conditions for the existence of a vector subspace $Y$ of $X$ such that every non-zero vector
 in $Y$ is both irregular for $T$  and distributionally near to zero for $T.$
\end{abstract}

\maketitle
\section{Introduction and preliminaries}

 Let $X$ be a Banach space and $B(X)$ denote the space
of all bounded linear operators $T : X \to X$.

\smallskip
Recall that an operator $T \in B(X)$ is said to be {\em Li-Yorke chaotic}
if there exists an uncountable set $\Gamma\subset X$ such that
for every pair $(x,y) \in \Gamma \times \Gamma$ of distinct points, we have
\[
 \liminf_{n \to \infty} \|T^n x - T^n y\| = 0
 \ \ \mbox{ and } \ \
 \limsup_{n \to \infty} \|T^n x - T^n y\| > 0.
\]
In this case, $\Gamma$ is called a {\em scrambled set} for $T$ and each
such pair $(x,y)$ is called a {\em Li-Yorke pair} for $T$. It was the first notion of chaos
in the mathematical literature, introduced by Li and Yorke \cite{LY} in the context of interval maps,  and it became very popular.

If the orbit of $x$ is simultaneously  near to $0$ and
unbounded, then $x$ is called a {\em
irregular vector} for $T$ (see \cite{bea-b}). It was proved in \cite{2011} that $T$ is
 Li-Yorke  chaotic if and only if $T$ admits an irregular vector.

An operator $T$ on a Banach space $X$ is said to be \emph{mean Li-Yorke
chaotic} if there is an uncountable subset $S$ of $X$ (a \emph{mean Li-Yorke
set} for $T$) such that every pair $(x,y)$ of distinct
points in $S$ is a \emph{mean Li-Yorke pair} for $T$, in the sense that
$$
\liminf_{N \to \infty} \frac{1}{N} \sum_{j=1}^N \|T^jx - T^jy\| = 0
\ \ \ \ \mbox{  and } \ \ \ \
\limsup_{N \to \infty} \frac{1}{N} \sum_{j=1}^N \|T^jx - T^jy\| > 0.
$$
%If $S$ can be chosen to be dense  in $X$, then we say that
%$T$ is \emph{densely mean Li-Yorke chaotic}.

Given an operator $T$ and a vector $x$, we say that $x$ is an
\emph{absolutely mean irregular}  \emph{vector}
for $T$ if
$$
\liminf_{N \to \infty} \frac{1}{N} \sum_{j=1}^N \|T^j x\| = 0
\ \ \ \ \mbox{  and } \ \ \ \
\limsup_{N \to \infty} \frac{1}{N} \sum_{j=1}^N \|T^j x\| = \infty .
$$

It was proved in \cite{ARXIV} that $T$ is mean
 Li-Yorke  chaotic if and only if $T$ admits an {absolutely mean irregular vector.

\smallskip
If $A \subset \Bbb N$, recall that the {\em lower} and the {\em upper
densities} of $A$ are defined as
$$
\ldens(A):= \liminf_{n \to \infty} \frac{\card(A \cap [1,n])}{n}
 \ \ \text{ and } \ \
\udens(A):= \limsup_{n \to \infty} \frac{\card(A \cap [1,n])}{n},
$$
respectively. Moreover, the {\it Banach upper density} of $A$ is given by
$$
\budens(A):= \lim _{n \to \infty} \limsup_{m \to \infty}
             \frac{\card (A \cap [m+1,m+n])}{n}\cdot
$$
and
the {\it Banach lower density} of $A$  by
$$
\bldens(A):= \lim _{n \to \infty} \liminf_{m \to \infty}
             \frac{\card (A \cap [m+1,m+n])}{n}\cdot
$$

Given $T \in B(X)$, $x ,y \in X$ and $\delta > 0$, let
\begin{equation}
F_{x,y}(\delta):= \ldens(\{j \in \Bbb N : \|T^j x - T^j y\| < \delta\})\label{ld}
\end{equation}
and
\begin{equation}
F_{x,y}^*(\delta):= \udens(\{j \in \Bbb N : \|T^j x - T^j y\| < \delta\}).\label{ud}
\end{equation}

If the pair $(x,y)$ satisfies

 {(DC1)} $F^*_{x,y} \equiv 1$ and $F_{x,y}(\eps) = 0$
              for some $\eps > 0$, or

 {(DC2)} $F^*_{x,y} \equiv 1$ and $F_{x,y}(\eps) < 1$
              for some $\eps > 0$, or

 {(DC2$\frac{1}{2}$)} There exist $c > 0$ and $r > 0$ such that
              $F_{x,y}(\delta) < c < F^*_{x,y}(\delta)$
              for all $ 0 < \delta < r$

then $(x,y)$ is called a {\em distributionally chaotic pair of type}
$k \in \{1,2,2\frac{1}{2}\}$ for $T$. The operator $T$ is said to be
{\em distributionally chaotic of type} $k$ if there exists an uncountable set
$\Gamma \subset X$ such that every pair $(x,y)$ of distinct points in
$\Gamma$ is a distributionally chaotic pair of type $k$ for $T$. In this case,
$\Gamma$ is a {\em distributionally scrambled set of type $k$} for $T$. This notion was introduced by Schweizer and Sm\'{\i}tal  \cite{SS94} for interval maps.

\smallskip
In the context of linear dynamics, it is usual to say that an operator
$T \in B(X)$ is distributionally chaotic if there exist an uncountable
set $\Gamma \subset X$ and an $\eps > 0$ such that for every pair $(x,y)$
of distinct points in $\Gamma$, we have that $F_{x,y}^* \equiv 1$ and
$F_{x,y}(\eps) = 0;$ as it is well known, this type of chaos is equivalent with distributional chaos of type $1$ for operators on Banach spaces  (see \cite{2011,2013JFA}, where further references
can be found).

\smallskip

We say that $T \in B(X)$ is {\em densely distributionally chaotic} if $\Gamma$ is  uncountable
{\em dense} set.

\smallskip
Let $T \in B(X)$ and $x \in X$. The orbit of $x$ is said to be
{\em distributionally near to $0$} if there exists $A \subset \Bbb N$ with
$\udens(A) = 1$ such that $\lim_{n \in A} T^nx = 0$. We say that $x$ has a
{\em distributionally unbounded orbit} if there exists $B \subset \Bbb N$ with
$\udens(B) = 1$ such that $\lim_{n \in B} \|T^nx\| = \infty$.
If the orbit of $x$ is simultaneously distributionally near to $0$ and
distributionally unbounded, then $x$ is called a {\em distributionally
irregular vector} for $T$. It was proved in \cite{2013JFA} that $T$ is
 distributionally chaotic if and only if $T$ admits a
distributionally irregular vector. A {\em distributionally irregular manifold}
for $T$ is a vector subspace $Y$ of $X$ such that every non-zero vector
$y$ in $Y$ is distributionally irregular for $T$.

\smallskip

Given $T \in B(X)$, $x ,y \in X$ and $\delta > 0$, let
\begin{equation}
BF_{x,y}(\delta):= \bldens(\{j \in \Bbb N : \|T^j x - T^j y\| < \delta\})\label{ld}
\end{equation}
and
\begin{equation}
BF_{x,y}^*(\delta):= \budens(\{j \in \Bbb N : \|T^j x - T^j y\| < \delta\}).\label{ud}
\end{equation}

\begin{prop}\label{zajeb}
Suppose that $X$ is a Banach space and $T\in L(X)$. If $T$  Li-Yorke chaotic, then there exist $x$ and subsets $A,B$ of $\Bbb N$ with  $\overline{Bd}(A)=\overline{Bd}(B)=1$ such that  $T^nx$ tends to zero with $n\in A$ and   $\|T^nx\|$ tends to $\infty$ with $n\in B$.
\end{prop}

\begin{proof}
Suppose that $T$ is Li-Yorke chaotic, then it is clear that $\|T\|>1.$
By \cite[Theorem 5]{2011}, we have the existence of a Li-Yorke irregular vector for $T,$ i.e., there exists a vector $x\in X$ such that
$\liminf_{n\rightarrow\infty}\| T^{n}x\|=0$ and $\limsup_{n\rightarrow\infty}\| T^{n}x\|=\infty.$ Therefore, there exist two strictly increasing sequences of positive integers $(n_{k})$ and $(l_{k})$ such that $\min(n_{k+1}-n_{k},l_{k+1}-l_{k})>k^{2},$ $\| T^{n_{k}}x\|<2^{-k^{2}}$ and $\| T^{l_{k}}x\|>2^{k^{2}}$ for all $k\in {\mathbb N}.$ Set $A:=\bigcup_{k\in {\mathbb N}}[n_{k},n_{k}+k]$ and
$B:=\bigcup_{k\in {\mathbb N}}[l_{k}-k,l_{k}].$ Then $\overline{Bd}(A)=\overline{Bd}(B)=1$ and for each $n\in A$ ($n\in B$) there exists
$k\in {\mathbb N}$ such that $n\in [n_{k},n_{k}+k]$ ($n\in [l_{k}-k,l_{k}]$) and therefore $\|T^{n}x\|\leq \|T\|^{k}2^{-k^{2}}$
($\|T^{n}x\|\geq \|T\|^{-k}2^{k^{2}}$). This completes the proof in a routine manner.
\end{proof}

 \begin{rem} If we change the upper and lower density  in the definition of distributional chaos, by Banach upper and Banach lower density, we thus obtain Li-Yorke chaos.
\end{rem}

\begin{defn}
 We say that $T$ is {\it reiteratively distributional chaotic of type $1$} if if there exists an uncountable set $\Gamma\subset X$ and some $\eps > 0$ such that
for every pair $(x,y) \in \Gamma \times \Gamma$ of distinct points $BF_{x,y}(\eps) = 0$
 and there exist $c > 0$ and $r > 0$ such that $c \le F^*_{x,y}(\delta)$ for all $ 0 < \delta < r$. If $F^*_{x,y}\equiv 1$, we say that $T$ is {\it reiteratively distributional chaotic of type $1^+$.}

We say that $T$ is {\it reiteratively distributional chaotic of type $2$} if  there exists an uncountable set $\Gamma\subset X$ and some $\eps > 0$ such that
for every pair $(x,y) \in \Gamma \times \Gamma$ of distinct points is $BF^*_{x,y}\equiv 1$  and $F_{x,y}(\eps) < 1$.

\end{defn}

Let us also recall that an operator $T \in B(X)$ is said to be
{\em frequently hypercyclic} (FH), {\em upper-frequently hypercyclic} (UFH),
{\em reiteratively hypercyclic} (RH) or {\em hypercyclic} (H) if there exists
a point $x \in X$ such that for every nonempty open subset $U$ of $X$, the set
$$
\{n \in \Bbb N : T^n x \in U\}
$$
has positive lower density, positive upper density, positive upper
Banach density or is nonempty, respectively. Such a point $x$ is called a
{\em frequently hypercyclic}, {\em upper-frequently hypercyclic},
{\em reiteratively hypercyclic} or {\em hypercyclic vector} for $T$,
respectively.
Moreover, $T$ is {\em mixing} if for every nonempty open sets $U,V \subset X$,
there exists $n_0 \in \Bbb N$ such that $T^n(U) \cap V \neq \emptyset$ for all
$n \geq n_0$;, $T$ is {\em weakly-mixing} if $T \oplus T$ is hypercyclic, and
$T$ is {\em Devaney chaotic} if it is hypercyclic and has a dense set of
periodic points. See \cite{BaGr06,bayart,2013JFA,BoGE07,erdper}, for instance.

\smallskip
Moreover, we would like to remind the readers of the famous
Frequent Hypercyclicity Criterion (FHC) \cite[Theorem 9.9]{erdper}.
This criterion  implies frequent hypercyclicity,  Devaney
chaos, mixing, distributional chaos and mean Li-Yorke  \cite{BaGr06, BoGE07, 2013JFA, erdper, ARXIV}.

This research report is organized as follows. In the section 2, we characterize reiterative distributional chaos of type 1, $1^+ $ and 2 in terms of existence of an irregular vector $x$ with additional properties. Moreover, we study relations with other dynamical properties.
In the section 3, we give conditions for existence of  a dense  reiteratively distributionally irregular manifold of type $1^+$ for $T$.

\section{Reiterative distributional chaos}

\begin{lem}\label{lemma3}
Let $T \in B(X)$. If $T$ admits a reiterative distributionally chaotic pair of type $
2,$ then $T$ admits a distributionally unbounded orbit.
\end{lem}

\begin{proof}
Since $T$ admits a reiterative distributionally chaotic pair of type 2,
there exist $u \in X$ and $\zeta > 0$ such that
$$
\xi:= \overline{dens}(\{j \in \Bbb N : \|T^j u\| > \zeta\}) > 0
$$
and there exists an increasing sequence $(n_k)$
in $\Bbb N$ such that $\|T^{n_k} u\| < 1/k$ for all $k \in \Bbb N$.

Let $y_k:= T^{n_k} u$ and $\eps:= \min\{\zeta,\frac{\xi}{2}\}$.
Then $\lim_{k \to \infty} y_k = 0$ and
$$
\overline{dens}(\{j \in \Bbb N : \|T^j y_k\| > \zeta\}) = \xi > \eps \ \ \ (k \in \Bbb N).
$$
Hence, there exists an increasing sequence $(N_k)$ in $\Bbb N$ such that
$$
card(\{1 \leq j \leq N_k : \|T^j y_k\| > \eps\}) \geq \eps N_k,
$$
for all $k \in \Bbb N$. Therefore, by \cite[Proposition 8]{2013JFA}, there exists
$y \in X$ with distributionally unbounded orbit.
\end{proof}

\begin{lem} \label{lemma4}
Let  $T\in B(X).$ Suppose that there exists a dense subset $X_0\subset X$ and a subset $A$ of $\Bbb N$  with  $\udens(A) = c>0$ such that each $x\in X_0$, $\displaystyle\lim_{n \in A} T^nx = 0$. Then the set of all vectors with orbits satisfying that $\displaystyle\lim_{n \in B} T^nx = 0$ with  $\udens(B) \ge c$ is residual.
\end{lem}

\begin{proof}
For $k,m\in\Bbb N $,  let
$$
 M_{k,m}=\left\{x\in X : \hbox{ there exists } n\in\Bbb N \hbox{ with }\card\{j\le n: \ \|T^jx\|<k^{-1}\}\ge n(c-k^{-1})\right\}.
 $$
Clearly $M_k$ is open and dense (since $M_k\supset X_0$). So the set
$X_1=\bigcap_{k,m} M_k$ is residual and consists of vectors with orbits  near of 0 in a set B  with  $\udens(B) \ge c>0$.
\end{proof}

\begin{thm}
\begin{itemize}
\item[(i)] $T$ is reiteratively distributional chaotic of type 1 if and only if  there exists an irregular vector $x$
and a subset $A$ of $\Bbb N$  with  $\udens(A) >0$ such that $\lim_{n \in A} T^nx = 0$.
\item[(ii)] $T$ is reiteratively distributional chaotic of type $1^+$ if and only if there exists an irregular vector $x$ with orbit distributionally near to zero.
\item[(iii)] $T$ is reiteratively distributional chaotic of type $2$ if and only if there exists an irregular vector $x$ with orbit distributionally unbounded.
\end{itemize}
\end{thm}
\begin{proof}

(i): It is clear (using ideas of the proof of Proposition \ref{zajeb}) that if there exists an irregular vector $x$ and a subset $A$ of $\Bbb N$  with  $\udens(A)>0$ such that $\lim_{n \in A} T^nx = 0$, then $T$ is reiteratively distributional chaotic of type $1$.

Suppose that $T$ is reiteratively distributional chaotic of type $1$, then there exists a pair $(x,y)$ of distinct points with $BF_{x,y}(\eps) = 0$
for some $\eps > 0$ and there exist $c > 0$ and $r > 0$ such that $c \le F^*_{x,y}(\delta)$ for all $ 0 < \delta < r$.
Let $u=x-y$ and consider
$$
Y_1 = \overline{span}(Orb(u,T)),
$$
which is an infinite dimensional closed $T$-invariant subspace of $X$. Consider the operator $S\in B(Y_1)$ obtained by restricting $T$ to $Y_1$.

Then  $S$ is  reiteratively distributional chaotic of type $1$ in $Y_1$.

Thus, by \cite[Corollary 5]{band}, $S$ has a residual set of points on $Y_1$ with orbit  unbounded. Moreover, by Lemma \ref{lemma4}, $S$ has a residual set of points on $Y_1$ with orbit  near to zero in a set with positive upper density.

Hence, $S$ has  an irregular vector $x$ with orbit  near of zero in a set with positive upper density on $Y_1$. Therefore, $T$ has  an irregular vector $x$ with orbit distributionally near of zero with positive upper density  on $X$.

(ii):  The proof is analogous to that of (i).

(iii): It is clear that if there exists an irregular vector $x$ with orbit distributionally unbounded, then $T$ is reiteratively distributional chaotic of type 2.

Suppose that $T$ is reiteratively distributional chaotic of type 2, then there exists a pair $(x,y)$ of distinct points with  $BF^*_{x,y}\equiv 1$  and $F_{x,y}(\eps) < 1$
              for some $\eps > 0$.
Let $u=x-y$ and consider $Y_2 = \overline{span}(Orb(u,T))$ and  the operator $S\in B(Y_2)$ obtained by restricting $T$ to $Y_2$.
Then  $S$ is  reiteratively distributional chaotic of type 2 in $Y_2$.

Thus, by Lemma \ref{lemma3} and \cite [Proposition 7]{2013JFA}, $S$ has a residual set of points on $Y_2$ with orbit distributionally unbounded. Moreover, by \cite[Corollary4]{band}, $S$ has a residual set of points on $Y_2$ with orbit near to zero.

Hence $S$ has  an irregular vector $x$ with orbit distributionally unbounded on $Y_2$. Therefore $T$ has  an irregular vector $x$ with orbit distributionally unbounded on $X$.
\end{proof}

Is clear from the definition that distributional chaos implies reiteratively distributional chaotic of type $1^+$ and type 2. Moreover

\begin{thm}
Any mean Li-Yorke chaotic operator is  reiteratively distributional chaotic of type $1^+$.
\end{thm}
\begin{proof}
Any mean Li-Yorke chaotic operator has an absolutely mean irregular vector. As a consequence of \cite[Proposition 20 (a)]{2018JMAA}, this is an irregular vector with orbit distributionally near of zero. Thus $T$ is  reiteratively distributional chaotic of type $1^+$.
\end{proof}

\begin{thm}
Any \emph{DC}$2\frac{1}{2}$ chaotic operator is  reiteratively distributional chaotic of type $1$ and of type $2$.
\end{thm}
\begin{proof}
It is sufficient to prove that any DC2$\frac{1}{2}$ chaotic operator has an irregular vector $x$ with orbit distributionally unbounded and  orbit  near to zero in a set with positive upper density.

If $T$  is {DC2$\frac{1}{2}$} chaotic, then  there exist a pair $(x,y)$ and $c > 0$ and $r > 0$ such that
$F_{x,y}(\delta) < c < F^*_{x,y}(\delta)$ for all $ 0 < \delta < r$.

Let $u=x-y$ and consider $Y_3 = \overline{span}(Orb(u,T))$
and the operator $S\in B(Y_3)$ obtained by restricting $T$ to $Y_3$.
Then  $S$ is  {DC2$\frac{1}{2}$} chaotic in $Y_3$.

Thus, by \cite[Lemma 7]{2018JMAA} and \cite [Proposition 7]{2013JFA}, $S$ has a residual set of points on $Y_3$ with orbit distributionally unbounded. Moreover,, by Lemma \ref{lemma4}, $S$ has a residual set of points on $Y_3$ with orbit  near to zero in a set with positive upper density.

Hence, $S$ has  an irregular vector $x$ with orbit distributionally unbounded and  orbit  near to zero in a set with positive upper density  on $Y_3$ . Therefore $T$ has an irregular vector $x$ with   the same property in $X$.
\end{proof}

\begin{cor}
Any upper-frequently hypercyclic  is
reiteratively distributional chaotic of type  1 and of type $2$.
\end{cor}
\begin{proof}
Any upper-frequently hypercyclic is a DC2$\frac{1}{2}$ chaotic operator \cite[Theorem 13]{2018JMAA}.
\end{proof}

\begin{example}\label{zjabb}
In \cite[Section 4]{BayaImre}, F. Bayart and I. Z. Ruzsa have constructed an invertible frequently hypercyclic operator $T=B_{{\bf w}}$ on the space $X:=c_{0}({\mathbb Z})$ that is not distributionally chaotic. This operator is clearly hypercyclic and therefore densely Li-Yorke chaotic, as well . Furthermore, the operator $B_{{\bf w}}$ has the property that there exists a set $A\subseteq {\mathbb N}$ such that $\underline{d}(A)>0,$ and for each $x=\langle x_{n}\rangle_{n\in {\mathbb N}} \in c_{0}({\mathbb Z})$ with $x_{k}\neq 0$ for some $k\in {\mathbb N},$ we have $\|B_{{\bf w}}^{n-k}x\|\geq |x_{k}|>0,$ whenever $n-k\in A.$ Thus $T$ is a frequently hypercyclic
(thus reiteratively distributional chaotic of type 1 and 2) that it is not  reiterative distributional chaotic of type $1^+$.
\end{example}

Given a hypercyclic operator $T$, in \cite{GM} is introduced the following quantity
$$
c(T) = sup_{x \in HC(T)}\udens \{ n\in \Bbb N: T^nx\in B(0, a)\}
$$
This quantity $c(T)\in  [0, 1]$ is associated to any
hypercyclic operator T, and essentially represents the maximal frequency with which the
orbit of a hypercyclic vector for $T$ can visit a ball in $X$ centered at 0.

It is shown in \cite{GM} that, in fact, for any $a$ one has $\udens \{ n\in \Bbb N: T^nx\in B(0, a)\}= c(T)$
 for a comeager set of hypercyclic vectors.

 \begin{thm}
Any  hypercyclic operator with $c(T)>0$ is
reiteratively distributional chaotic of type 1.
\end{thm}

\begin{proof}
As $T$ is hypercyclic, there exists a residual set $C_1$ of points with orbit unbounded.

Moreover, as a consequence of \cite[Proposition 4.7]{GM}, there exists a residual set
$C_2$ of vectors $y \in X$ such that $\|T^j y\| \to 0$ as $j \to \infty$
along some set $E_y$ with $\udens(E_y) = c(T)$.

Let $z \in C_1 \cap C_2$.  Then $z$ is  an irregular vector satisfying
 $\lim_{n \in A} T^nx = 0$ in a subset $A$ of $\Bbb N$  with  $\udens(A) >0$, thus
 $T$  is reiteratively distributional chaotic of type 1.
\end{proof}

\begin{figure}[h]
%\hspace*{1cm}
\begin{tikzpicture}[scale=0.22,>=stealth]
 \node[right] at (18,21) {FHC};
  \node[right] at (-3,15) {mixing};
 \node[right] at (5,15) {Devaney chaotic};
  \node[right] at (29,15) {FH};
  \node[right] at (33,15) {DC1$\equiv$ DC2};
  \node[right] at (45,15) {mean Li-Yorke};
  \node[right] at (28,10) {UFH};
 \node[right] at (19,5) {RH};
\node[right] at (34,5) {DC2$\frac{1}{2}$};
 \node[right] at (28,0) {RDC2};
 \node[right] at (44,5) {RDC$1^+$};
  \node[right] at (40,0) {RDC1};
 %\node[right] at (45,5) {mean Li-Yorke};
  \node[right] at (17,0) {w-mixing};
 \node[right] at (19,-5) {H};
  \node[right] at (15,-10) {Li-Yorke chaotic};
  %  \node[right] at (38,-10) {DC3};
 \draw[double, ->] (20,20) -- (0,16);
  \draw[double, ->] (20,20) -- (10,16);
   \draw[double, ->] (20,20) -- (30,16);
  \draw[double, ->] (20,20) -- (37,16);
   \draw[double, ->] (20,20) -- (50,16);
   % \draw[double, ->] (50,14) -- (50,6);
      \draw[double, ->] (50,14) -- (46,6);
       \draw[double, ->] (46,4) -- (43,1);
       \draw[double, ->] (42,-1) -- (20,-9);
       \draw[double, ->] (37,14) -- (46,6);
   \draw[double, ->] (10,14) -- (20,6);
 \draw[double, ->] (0,14) -- (19,1);
  \draw[double, ->] (30,14) -- (30,11);
  \draw[double, ->] (30,9) -- (37,6);
  \draw[double, ->] (30,9) -- (20,6);
   \draw[double, ->] (37,14) -- (37,6);
\draw[double, ->] (20,4) -- (20,1);
\draw[double, ->] (20,-1) -- (20,-4);
\draw[double, ->] (37,4) -- (31,1);
\draw[double, ->] (37,4) -- (43,1);
\draw[double, ->] (30,-1) -- (20,-9);
\draw[double, ->] (20,-6) -- (20,-9);
%\draw[double, ->] (40,4) -- (40,-9);
 \end{tikzpicture}
\caption{Implications between different definitions.} % related with hypercyclicity and caos in Banach spaces.}
\end{figure}
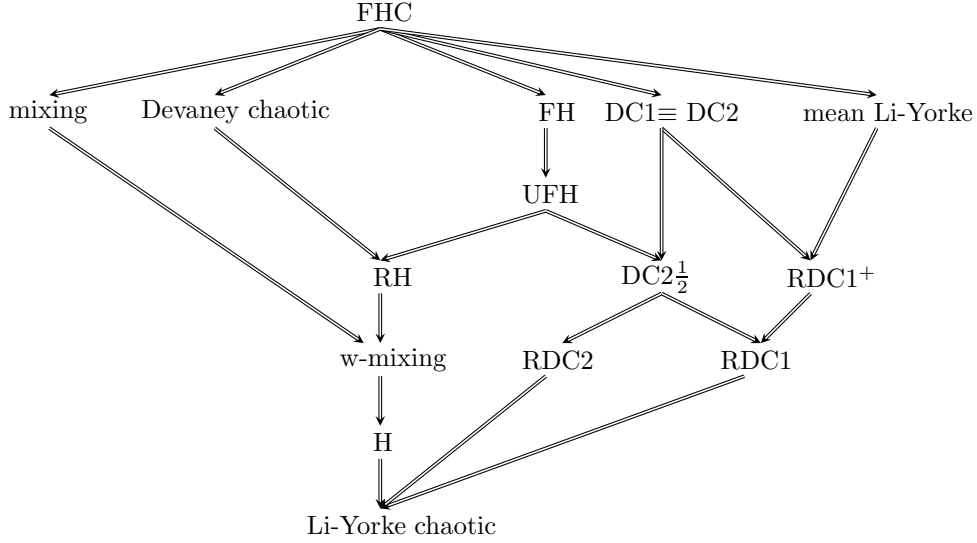

In the continuation, we will see that under the condition that there exists a dense set
$X_0 \subset X$ such that $T^nx \to 0$ for all $x \in X_0$,
$T$ is reiteratively distributionally chaotic of type $2$ is  equivalent to DC1.

\begin{thm}\label{theorem3}
Let $T \in B(X)$ and assume that there exists a dense set $X_0 \subset X$
such that
$$
T^nx \to 0 \ \ \text{ for all } x \in X_0.
$$
Then the following assertions are equivalent:
\begin{itemize}
\item [\rm (i)] $T$ is distributionally chaotic of type $1$;
\item [\rm (ii)]  $T$ is distributionally chaotic of type $2$;
\item [\rm (iii)]   $T$ is distributionally chaotic of type $2\frac{1}{2}$;
\item [\rm (iv)]  $T$ is reiteratively distributionally chaotic of type $2$.
\end{itemize}
\end{thm}

\begin{proof}
It is enough to prove that (iv) $\Rightarrow$ (i). So, assume (iv).
The proof of \cite[Theorem 15]{2013JFA} shows that if an operator $T$
has a distributionally unbounded orbit and $T^nx \to 0$ for every $x$
in a dense subset of $X$, then $T$ has a distributionally irregular vector.
Hence, in view of our hypotheses and Lemma \ref{lemma3}, we conclude that
(i) holds.
\end{proof}

\begin{cor}
Let $X$ be a Banach sequence space in which $(e_n)$ is a basis (\cite{erdper},
Section 4.1). Suppose that the unilateral weighted backward shift
$$
B_w(x_1,x_2,x_3,\ldots):= (w_2x_2,w_3x_3,w_4x_4,\ldots)
$$
is an operator on $X$. Then properties (i)--(iv) in the above theorem are
equivalent for the operator $B_w$.
\end{cor}

\begin{proof}
Since the orbit of any sequence with finite support converges to zero,
we can apply the above theorem.
\end{proof}

By the other hand, we will see that under the condition that there exists a dense set
$X_0 \subset X$ such that $T^nx \to 0$ for all $x \in X_0$, Li-Yorke chaos is equivalent to
$T$ is reiteratively distributionally chaotic of type $1^+$.

\begin{thm}\label{theorem3}
Let $T \in B(X)$ and assume that there exists a dense set $X_0 \subset X$
such that
$$
T^nx \to 0 \ \ \text{ for all } x \in X_0.
$$
Then the following assertions are equivalent:
\begin{itemize}
\item [\rm (i)] $T$ is Li-Yorke  chaotic;
\item [\rm (ii)]$T$ is reiteratively distributionally chaotic of type $1^+$.
\end{itemize}
\end{thm}

\begin{proof}
It is enough to prove that (i) $\Rightarrow$ (ii). So, assume (i). Then
by  \cite[Proposition 9]{2013JFA}, if there exists a dense set $X_0 \subset X$
such that $T^nx \to 0 \ \ \text{ for all } x \in X_0$, then the set of points with orbit distributionally near to zero is residual. Moreover, $T$ is Li-Yorke chaotis, thus the set of points with unbounded orbit is residual.
Hence there exist a point with orbit unbounded and orbit distributionally near to zero. As a consequence, $T$ is reiteratively distributionally chaotic of type $1^+$.
\end{proof}
\begin{cor}
Let $X$ be a Banach sequence space in which $(e_n)$ is a basis. Suppose that the unilateral weighted backward shift
$$
B_w(x_1,x_2,x_3,\ldots):= (w_2x_2,w_3x_3,w_4x_4,\ldots)
$$
is an operator on $X$. Then they are equivalents:
\begin{itemize}
\item [\rm (i)] $B_w$ is not power bounded;
\item [\rm (ii)]$B_w$ is reiteratively distributionally chaotic of type $1^+$.
\end{itemize}

\end{cor}

\begin{figure}[h]
%\hspace*{-1cm}
\begin{tikzpicture}[scale=0.3,>=stealth]
  \node[right] at (10,20) {{\bf Remark:} If $\{ x: T^nx \rightarrow 0\}$ is a dense set of $X$, then: };
  \node[right] at (37,0) {DC1$\equiv $ RDC2};
  \node[right] at (25,15) {Devaney chaotic};
  \node[right] at (28,10) {UFH};
  %\node[right] at (45,10) {Cesaro Hyper.};
  \node[right] at (37,-5) {Mean Li-Yorke};
    %\node[right] at (42,-10) {Mean Li-Yorke chaotic };
  \node[right] at (19,5) {RH};
  \node[right] at (17,0) {w-mixing};
 \node[right] at (19,-5) {H};
  \node[right] at (28,-15) {Li-Yorke$\equiv$ RDC$1^+$};
  \draw[double, ->] (31,9) -- (40,1);
  \draw[double, ->] (30,9) -- (20,6);
  %\draw[double, ->] (50,9) -- (50,-4);
  \draw[double, ->] (40,-1) -- (40,-4);
  %\draw[double, ->] (52,-6) -- (30,-14);
   \draw[double, ->] (30,14) -- (20,6);
  \draw[double, ->] (20,4) -- (20,1);
  \draw[double, ->] (30,14) -- (40,1);
\draw[double, ->] (20,-1) -- (20,-4);
\draw[double, ->] (20,-6) --(30,-14);
%\draw[double, ->] (40,4) -- (40,-4);
%\draw[double, ->] (50,-6) -- (50,-9);
\draw[double, ->] (40,-6) -- (30,-14);
 \end{tikzpicture}
\caption{Implications between the different definitions
 when  $\{x: T^nx \rightarrow 0\}$ is a dense set of $X$.}
\end{figure}
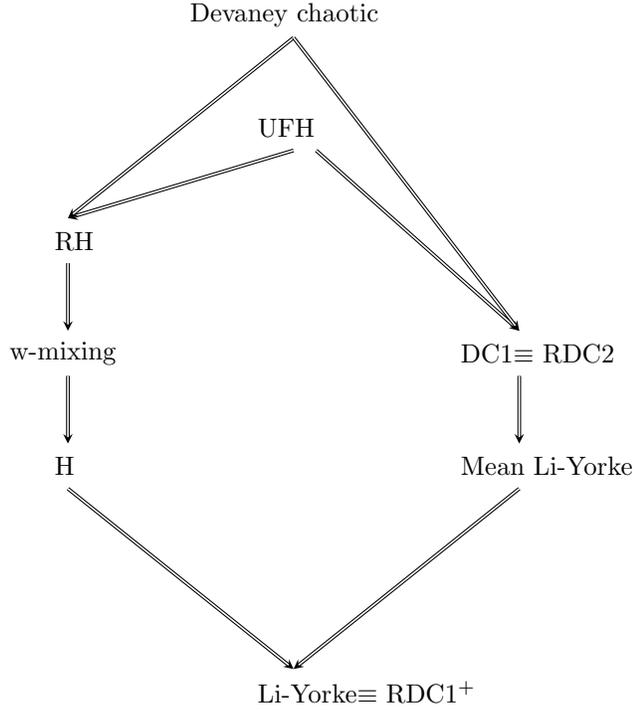

Recall that T. Berm\'udez. et al. (\cite[Theorem 2.1] {cesaro} prove that the operator $T$ in $l_{p}({\mathbb N})$ given by $Te_{1}:=0$ and $Te_{k}:=(k/k-1)^{\alpha}e_{k-1}$ for all $k>1$ with $0<\alpha <\frac{1}{p}$ (where $(e_{k})_{k\in {\mathbb N}}$ is the standard basis of $l_{p}({\mathbb N})$) is absolutely Cesaro bounded (that is, there
exists a constant $C > 0$ such that
$\sup_{N \in \Bbb N} \frac{1}{N} \sum_{j=1}^N \|T^j x\| \leq C \|x\|$, $\forall  x \in X$)
and not power bounded. Moreover since this backward shift  is mixing, we have that $T$ is  Li-Yorke chaotic and since
is absolutely Cesaro bounded, $T$ cannot be distributionally chaotic  \cite[Proposition 20 b)]{2018JMAA}. It is natural to ask whether $T$ is reiteratively distributionally chaotic of type $1$ or $2$?.

\begin{example}\label{problemen}
The operator $T$ defined above is mixing, absolutely Cesaro bounded,  reiteratively distributionally chaotic of type $1^+$ and not reiteratively distributionally chaotic of type $2$.
\end{example}

\begin{proof}

$T$ is reiteratively distributionally chaotic of type $1^+$, because it is clearly Li-Yorke chaotic.

$T$ is not reiteratively distributionally chaotic of type $2$ because is absolutely Cesaro bounded and thus it does not admit an orbit distributionally unbounded.
\end{proof}

\begin{rem}
The operator in the above example is not mean Li-Yorke chaotic, since it is absolutely Cesaro bounded.
\end{rem}

Also, the linear continuous operator $B$ considered by F. Mart\'inez-Gim\'enez, P. Oprocha and A. Peris in \cite[Theorem 2.1]{gimenez} is topologically mixing and not distributionally chaotic (the state space in their analysis is a weighted $l^{p}$-space). Furthermore, for each $x\in X$ and $\epsilon>0$ the  density of set $\{n\in {\mathbb N} : \|B^{n}x\|\leq \epsilon\}$ is equal to one and this immediately implies that the operator $B$ cannot be reiteratively distributionally chaotic of type $2$.

\begin{ques}
Exists there mixing operator, no   reiterative distributional chaotic of type  1  on Banach spaces?
\end{ques}

 The example 6.35 of \cite{GMM} is a mixing and chaotic operator such that  for any hypercyclic vector $x$, there exists an $\varepsilon>0$ such that $\underline{dens}\{n\in {\mathbb N} : \|T^{n}x\|\geq \epsilon\}=1$, thus $T$ does not have a dense set of points with orbit near to zero in a set of upper density $C>0$ ( since then this set should contain a hypercyclic vector  because  is residual ).  Is this operator reiteratively distributional chaotic of type 1 or type $1^+$?

\bigskip

Let us recall that by Q. Menet's result \cite[Theorem 1.2]{menet}, there exists a linear continuous operator $T$ on the Banach space $l^{p}$ ($1\leq p<\infty$) or $c_{0}$ that is chaotic (and therefore reiteratively hypercyclic), not distributionally chaotic and not upper frequently hypercyclic.
This operator has the property that for each $x\in X \setminus \{0\}$ there exists an $\varepsilon>0$ such that $\underline{dens}\{n\in {\mathbb N} : \|T^{n}x\|\geq \epsilon\}>0$, thus $T$ is not reiterative distributional chaotic of type $1^+$.

However, the Menet operator  is reiterative distributional chaotic of type  2. Thus, it is natural to ask:

\begin{ques}
Exists there chaotic operator, no   reiterative distributional chaotic of type  2  on Banach spaces?
\end{ques}

Before proceeding further, we would like to note that H. Bingzhe and L. Lvlin have shown in \cite[Theorem 2.1]{melvin} that there exists an invertible bounded linear operator $T$ on an infinite-dimensional Hilbert space $H$ such that $T$ is both distributionally chaotic  but $T^{-1}$ is not Li-Yorke chaotic. See also \cite[Section 6]{2013JFA}, where it has been constructed a densely distributionally chaotic operator $T\in L(l^{1}({\mathbb Z}))$ such that $T$ is invertible and $T^{-1}$ is not distributionally chaotic; arguing as in the proof given after Example \ref{problemen} below, we can show that $T^{-1}$ is densely reiteratively distributionally chaotic of type $1^+$ and not reiteratively distributionally chaotic of type $2$.

\section{Dense reiteratively distributionally irregular manifold}\label{bruks}

A {\em reiteratively distributionally irregular manifold of type $1^+$}
for $T$ is a vector subspace $Y$ of $X$ such that every non-zero vector
$y$ in $Y$ is irregular for $T$  and distributionally near to zero for $T$.

\begin{thm}\label{na-dobro}
Suppose that $X$ is separable, $T\in B(X)$,
\begin{itemize}
\item[(i)] $X_0:= \Big\{x \in X : \lim _{n \to \infty} \|T^n x\| = 0\Big\}$ is dense in $X$,
and
\item[(ii)] $T$ is not power bounded.
\end{itemize}
Then there exists a dense  reiteratively distributionally irregular manifold of type $1^+$ for $T$.
\end{thm}

\begin{proof}
Under hypothesis of theorem,  $C:= \|T\| > 1$ and we can choose a sequence $(x_m)$ of
normalized vectors in $X_0$ and an increasing sequence $(N_m)$ of positive
integers so that
\begin{align*}
 &\|T^{N_m}x_m\| > m(2C)^m \ \mbox{,} \
 \|T^{N_m}x_k\| < \frac{1}{m} \ \mbox{  for }
k = 1,\ldots,m-1\  \\ & \mbox{ and } \
\frac{1}{N_m} \sum_{i=1}^{N_m} \|T^ix_k\| < \frac{1}{m} \ \mbox{  for }
k = 1,\ldots,m-1.
\end{align*}
Given $\alpha,\beta \in \{0,1\}^\Bbb N$, we say that $\beta\leq \alpha $
if $\beta_i\leq \alpha_i$ for all $i \in \Bbb N$. Let $(r_j)$ be a sequence
of positive integers such that $r_{j+1} \geq 1 + r_j + N_{r_j + 1}$ for all
$j \in \Bbb N$. Let $\alpha \in \{0,1\}^\Bbb N$ be defined by $\alpha_n = 1$ if and
only if $n = r_j$ for some $j \in \Bbb N$. For each $\beta \in \{0,1\}^\Bbb N$ such
that $\beta \leq \alpha$ and $\beta$ contains an infinite number of
$1$'s, we define
$$
x_\beta := \sum_i \frac{\beta_i}{(2C)^i}\, x_i
         = \sum_j \frac{\beta_{r_j}}{(2C)^{r_j}}\, x_{r_j}.
$$
Take $k \in \Bbb N$ with $\beta_{r_k} = 1$. Since
$\|T^{N_{r_k}} x_{r_k}\| >
r_k (2C)^{r_k}$ and
$\displaystyle \|T^N_{r_k} x_s\| <
\frac{1}{r_k}$ for each $s < r_k$, we have that
\begin{align*}
\norm{T^{N_{r_k}}x_\beta}
 &\ge \norm{T^{N_{r_k}} x_{r_k}}
      - \sum_{j \neq k} \frac{\beta_{r_j}}{(2C)^{r_j}} \norm{T^{N_{r_k}} x_{r_j}} \\
 &>   r_k - \frac{1}{r_k} \sum_{j < k} \frac{1}{(2C)^{r_j}}
      - \sum_{j > k} \frac{\norm{x_{r_j}}}{2^{r_j}}
 \ge r_k - 1.
\end{align*}
On the other hand,  since
$\displaystyle \frac{1}{N_{r_k+1}}\sum_{i=1}^{N_{r_k+1}}\norm{T^i x_s} <
\frac{1}{r_k + 1}$ for each $s < r_k + 1$, then
\begin{align*}
 \frac{1}{N_{r_k+1}}\sum_{i=1}^{N_{r_k+1}}\norm{T^i x_\beta}
 &\le  \frac{1}{N_{r_k+1}}\sum_{i=1}^{N_{r_k+1}}\sum_{j \le k} \frac{\beta_{r_j}\|T^i x_{r_j}\|}{(2C)^{r_j}}
    +  \frac{1}{N_{r_k+1}}\sum_{i=1}^{N_{r_k+1}}\sum_{j > k} \frac{\beta_{r_j}\|T^i x_{r_j}\|}{(2C)^{r_j}}\\
 &\le \frac{1}{r_k + 1} \sum_{j \le k} \frac{1}{(2C)^{r_j}}
    +  \frac{1}{N_{r_k+1}}\sum_{i=1}^{N_{r_k+1}}\sum_{j > k} \frac{\|x_{r_j}\|}{2^{r_j}}
 <   \frac{1}{r_k + 1}\cdot
\end{align*}
Thus, $x_\beta$ is irregular for $T$ and $x_\beta$ is distributionally near to zero for $T$.

Now, let $(w_n)$ be a dense sequence in $X_0$ and choose
$\gamma_n \in \{0,1\}^\Bbb N$ ($n \in \Bbb N$) such that each $\gamma_n$
contains an infinite number of $1$'s, $\gamma_n \leq \alpha$ for every
$n \in \Bbb N$, and the sequences $\gamma_n$ have mutually disjoint supports.
Define $v_n := \sum_i \frac{\gamma_{n,i}}{(2C)^i}\, x_i$ and
$y_n:= w_n + \frac{1}{n}\, v_n$ ($n \in \Bbb N$).
Then $Y:= \spa\{y_n : n \in \Bbb N\}$ is a dense subspace of $X$.
Moreover, if $y \in Y \backslash \{0\}$, then we can write
$y = w_0 + \sum_k \frac{\rho_k}{(2C)^k}\, x_k$,
where $w_0 \in X_0$ and the sequence of scalars $(\rho_k)$ takes
only a finite number of values (each of them infinitely many times).
As in the above proof, we can show that the vector
$v := \sum_k \frac{\rho_k}{(2C)^k}\, x_k$
is  irregular for $T$ and distributionally near to zero  for $T$. Since $y = w_0 + v$ and $w_0 \in X_0$,
we conclude that $y$ is also  irregular for $T$ and distributionally near to zero  for $T$.
\end{proof}

\begin{cor}
Let $X$ be a Banach sequence space in which $(e_n)$ is a basis. Suppose that the unilateral weighted backward shift
$$
B_w(x_1,x_2,x_3,\ldots):= (w_2x_2,w_3x_3,w_4x_4,\ldots)
$$
is an operator on $X$. Then they  are equivalents:
\begin{itemize}
\item [\rm (i)] $B_w$ is not power bounded;
\item [\rm (ii)]$B_w$ has a  dense  reiteratively distributionally irregular manifold of type $1^+$.
\end{itemize}

\end{cor}

\vspace{.1in}

\begin{thebibliography}{90}



\bibitem{BaGr06} {\sc F. Bayart, S. Grivaux,}
    {\it Frequently hypercyclic operators},
    Trans.\ Amer.\ Math.\ Soc.\ {\bf 358} (2006), no.\ 11, 5083--5117.

\bibitem{bayart}
\textsc{F. Bayart, E. Matheron,}
\emph{Dynamics of Linear Operators,}
Cambridge Tracts in Mathematics, Cambridge University
Press, Cambridge, UK, \textbf{179(1)}, 2009.


\bibitem{BayaImre}
\textsc{ F. Bayart, I. Z. Ruzsa,}
\emph{ Difference sets and frequently hypercyclic weighted shifts,}
Ergodic Theory Dynam. Systems, {\textbf 35} (2015), 691--709.



\bibitem{bea-b}
\textsc{ B. Beauzamy,}
\emph{ Introduction to Operator Theory and Invariant Subspaces,} North-Holland, Amsterdam,
1988.

\bibitem{2011}
\textsc{T. Berm\'udez, A. Bonilla, F. Martinez-Gim\'enez, A. Peris,}
\emph{Li-Yorke and distributionally chaotic operators,}
J. Math. Anal. Appl. \textbf{373} (2011),  83--93.

\bibitem{cesaro}
\textsc{T. Berm\'udez, A. Bonilla,  V. M\"uller, A. Peris,}
\emph{Ce\`saro bounded operators in Banach spaces,}
 to appear in J. d'Analyse Math.



\bibitem{2013JFA}
\textsc{N. C. Bernardes Jr., A. Bonilla, V. M\"uller, A. Peris,}
\emph{Distributional chaos for linear operators,} J. Funct. Anal.
\textbf{265} (2013),  no. 1, 2143--2163.

\bibitem{2018JMAA}
\textsc{N. C. Bernardes Jr., A. Bonilla, A. Peris, X. Wu,}
\emph{Distributional chaos for operators on Banach spaces,}
J. Math. Anal. Appl. {\bf 459} (2018), 797--821.

\bibitem{band}
\textsc{ N. C. Bernardes Jr, A. Bonilla, V. M\"uller, A. Peris,}
\emph{ Li-Yorke chaos in linear dynamics,}
Ergodic Theory Dynamical Systems
35 (2015), 1723--1745.

\bibitem{ARXIV}
\textsc{N. C. Bernardes Jr., A. Bonilla, A. Peris,}
\emph{Mean Li-Yorke chaos in Banach spaces,}
preprint. https://arxiv.org/pdf/1804.03900.





\bibitem{BoGE07}
\textsc{A. Bonilla and K.-G. Grosse-Erdmann},
   \emph{ Frequently hypercyclic operators and vectors},
    Ergodic Theory Dynam.\ Systems {\bf 27} (2007), no.\ 2, 383-404.
    Erratum: Ergodic Theory Dynam.\ Systems {\bf 29} (2009), no.\ 6, 1993-1994.



\bibitem{mendoza}
\textsc{J. A. Conejero, M. Kosti\' c, P. J. Miana, M. Murillo-Arcila,}
\emph{Distributionally chaotic families of operators on Fr\' echet spaces,}
Comm. Pure Appl. Anal. {\bf 15} (2016), 1915--1939.





\bibitem{GM}
 \textsc{ S. Grivaux, \'E. Matheron},
    \emph{ Invariant measures for frequently hypercyclic operators},
    Adv.\ Math.\ {\bf 265} (2014), 371--427.

    \bibitem{GMM}
    \textsc{S. Grivaux, \'E. Matheron, Q. Menet},
    \emph{ Linear dynamical systems on Hilbert spaces: typical properties and explicit examples},
 Memoirs of Amer. Math. Soc, to appear

\bibitem{erdper}
\textsc{K.-G. Grosse-Erdmann, A. Peris,}
\textit{Linear Chaos},
Springer-Verlag, London, 2011.



\bibitem{LY}
\textsc{ T. Y. Li, J. A. Yorke,}
{\it Period three implies chaos,}
Amer. Math. Monthly. 2 (1975), 985--992.



\bibitem{melvin}
\textsc{L. Luo, B. Hou},
\emph{Li-Yorke chaos for invertible mappings on non-compact spaces,}
Turkish J. Math. {\bf 40} (2016), 411--416.

\bibitem{luo}
\textsc{B. Hou, L. Luo},
\emph{ Some remarks on distributional chaos for bounded linear operators,}
Turkish J. Math. {\bf 39} (2015), 251--258.



\bibitem{gimenez}
\textsc{F. Mart\'inez-Gim\'enez, P. Oprocha, A. Peris,}
\emph{Distributional chaos for operators with full scrambled sets,}
Math. Z. \textbf{274} (2013), 603--612.

\bibitem{menet}
\textsc{Q.  Menet,}
\emph{Linear  chaos  and  frequent  hypercyclicity,} Trans. Amer. Math. Soc. \textbf{369} (2017), 4977--4994.


\bibitem{SS94} B. Schweizer and J. Sm\'ital,
  \emph{ Measures of chaos and a spectral decomposition of dynamical systems on the interval},
    Trans.\ Amer.\ Math.\ Soc.\ {\bf 344} (1994), no.\ 2, 737--754.


\end{thebibliography}
\end{document}